\documentclass[twoside, 11pt]{article}

\usepackage[a4paper, left = 1in, right = 1in, textheight = 22cm]{geometry}

\usepackage{enumerate}

\usepackage{amsmath, amsfonts, amssymb, mathtools, amsthm}

\usepackage{hyperref}
\usepackage{cleveref}

\usepackage{biblatex}
\DeclarePairedDelimiter\abs{\lvert}{\rvert}
\DeclareMathOperator{\Span}{span}

\newcommand{\dis}{\displaystyle}
\renewcommand{\leq}{\leqslant}
\renewcommand{\geq}{\geqslant}

\newcommand{\R}{\mathbb{R}}

\newcommand{\C}{\mathbb{C}}
\newcommand{\N}{\mathbb{N}}
\renewcommand{\P}{\mathbb{P}}

\newcommand{\e}{\varepsilon}

\theoremstyle{plain}
\newtheorem{theorem}{Theorem}[section]
\newtheorem{proposition}[theorem]{Proposition}
\newtheorem{lemma}[theorem]{Lemma}

\theoremstyle{definition}
\newtheorem{definition}[theorem]{Definition}

\crefname{proposition}{Proposition}{Propositions}
\crefname{lemma}{Lemma}{Lemmas}
\crefname{theorem}{Theorem}{Theorems}

\makeatletter
\let\oldabs\abs
\def\abs{\@ifstar{\oldabs}{\oldabs*}}

\title{\textbf{Double algebraic genericity of universal harmonic functions on trees}}
\author{C. A. Konidas}
\date{\vspace{-5ex}}

\begin{document}
\pagestyle{myheadings}
\markboth{Double algebraic genericity of universal functions on trees}{C. A. Konidas}
\maketitle
\begin{abstract}
\noindent It is well known that the set of universal functions on a tree contains a vector space except zero which is dense in the set of harmonic functions.
In this paper we improve this result by proving that the set of universal functions on a tree contains two vector spaces except zero which are dense in the space of harmonic functions and intersect only at zero.
\end{abstract}
{\em AMS classification numbers}: 05C05, 60J45, 60J50, 30K99, 46M99\smallskip

\section{Introduction}

Let $T$ be the set of vertices of a rooted tree with root $x_0$.
For each $n \in \N$ we define $T_n$ to be the set of all vertices at distance $n$ from $x_0$ and we also define $T_0 = \{x_0\}$.
Given a vertex $x \in T$, we shall write $S(x)$ for the set of the children of x.
We assume that all $T_n$ are finite and for every $x \in T$ the set $S(x)$ has at least two elements.
Since $T = \bigcup_{n=0}^{\infty}{T_n}$ and each $T_n$ is finite and nonempty, the set $T$ is infinite denumerable.
With each $x \in T$ and $y \in S(x)$ we associate a real number $q(x,y) > 0$, which we think of as the probability of transition from vertex $x$ to vertex $y$, such that
$$\sum_{y \in S(x)}{q(x,y)} = 1.$$

We define the boundary of $T$, denoted by $\partial T$, as the set of all infinite geodesics originating from $x_0$.
More specifically, an element $e \in \partial T$ is of the form $e = \{z_n \in T: n \in \N\}$, where $z_1 = x_0$ and $z_{n+1} \in S(z_n)$ for all $n \in \N$.
For each $x \in T$ we define the boundary sector of $x$ as $B_x = \{ e\in \partial T: x \in e\}$.
Notice that $\{B_x: x \in T_n\}$ is a partition of $\partial T$ for all $n \in \N$.
For each $x \in T\setminus \{x_0\}$, we have that $x \in T_n$ for some $n \in \N$ and we assign probability $p(B_x) = \prod_{j=0}^{n-1}{q(y_j,y_{j+1})}$ where $y_0 = x_0, y_n = x$ and $y_{j+1} \in S(y_j)$ for all $j \in \{1,\dots,n\}$.
For each $n \in \N$ we consider $\mathcal{M}_n$ to be the $\sigma$-algerba on $\partial T$ generated by $\{B_x: x  \in T_n\}$, which is a finite partition of $\partial T$ and therefore a function $h:\partial T \to \C$ is $\mathcal{M}_n$-measurable if and only if it is constant on every $B_x$ for $x \in T_n$.
We can extend $p(B_x)$ to a probability measure $\P_n$ on the measurable space $(\partial T, \mathcal{M}_n)$.
One can easily check that $\mathcal{M}_n \subseteq \mathcal{M}_{n+1}$ and $\P_{n+1}|_{\mathcal{M}_n} = \P$ for all $n \in \N$.
By Kolmogorov's Consistency Theorem there exists a probability measure $\P$ on the measurable space $(\partial T, \mathcal{M})$, where $\mathcal{M}$ is the completion of the $\sigma$-algebra on $\partial T$ generated by the set $\bigcup_{n=1}^{\infty}{\mathcal{M}_n}$, such that $\P|_{\mathcal{M}_n} = \P_n$ for all $n \in \N$.
A function $h: \partial T \to \C$ is $\mathcal{M}$-measurable if and only if for every open set $V \subseteq \C$ we have $h^{-1}(V) \in \mathcal{M}$.
By identifying almost everywhere equal $\mathcal{M}$-measurable functions we obtain the space $L^0(\partial T, \C)$ or $L^0(\partial T)$ which we endow with the metric
$$P(\psi,\phi) = \int_{\partial T}{\frac{\abs{\psi(e) - \phi(e)}}{1+\abs{\psi(e) - \phi(e)}}} \, d\P(e),$$
where $\psi, \phi \in L^0(\partial T)$.
The induced topology is that of convergence in probability.

We consider the space $\C^{T}$ endowed with the metric
$$\rho(f,g) = \sum_{n=1}^{\infty}{\frac{1}{2^n}\frac{\abs{f(z_n) - g(x_n)}}{1+\abs{f(z_n) - g(z_n)}}},$$
where $f,g \in \C^{T}$ and $\{z_n: n \in \N\}$ is a enumeration of $T$.
The induced topology is that of pointwise convergence.
A function $f \in \C^T$ is called harmonic if and only if for all $x \in T$ it holds that
$$f(x) = \sum_{y \in S(x)}{q(x,y)f(y)}.$$
The set of all harmonic functions shall be denoted by $H(T,\C)$ or $H(T)$.
For every $f \in H(T)$ and $n \in \N$ we define a function $\omega_n(f): \partial T \to \C$ by $\omega_n(f)(e) = f(z)$, where $z$ is the unique element of $e \cap T_n$ and $e \in \partial T$.
Notice that $\omega_n(f)$ is constant on each $B_x$ for $x \in T_n$ and thus it is $\mathcal{M}_n$-measurable.
A harmonic function $f \in H(T)$ is called universal if and only if the sequence $\{\omega_n(f)\}_{n=1}^{\infty}$ is dense in $L^0(\partial T)$. The set of universal functions shall be denoted by $U(T,\C)$ or $U(T)$.
The following theorem was proved in \cite{BIEHLER2}.

\begin{theorem}
The set $U(T) \cup \{0\}$ contains a vector space which is dense in $H(T)$, that is we have algebraic genericity for the set $U(T)$.
\end{theorem}

In this paper we will improve this result by proving the following theorem.

\begin{theorem}
\label{thm:wanted}
There exist two vector spaces $F_1,F_2$ contained in $U(T)\cup\{0\}$ which are dense in $H(T)$ such that $F_1 \cap F_2 = \{0\}$, that is we have double algebraic genericity for the set $U(T)$.
\end{theorem}

In order to prove this we will consider two particular subsets of $U(T)$ which have been studied before in \cite{ABAKUMOV2,ABAKUMOV1,BIEHLER2,BIEHLER1}.
We say that a harmonic function $f:T \to \C$ is frequently universal if and only if for every nonempty open set $V \subseteq L^0(\partial T)$ the set $\{n \in \N: \omega_n(f) \in V \} $ has strictly positive lower density.
The set of frequently universal functions shall be denoted by $U_{FM}(T,\C)$ or $U_{FM}(T)$.
We say that a harmonic function $f:T \to \C$ belongs to the class $X(T,\C)$ or $X(T)$ if and only if for every nonempty open set $V \subseteq L^0(\partial T)$ the set $\{n \in \N: \omega_n(f) \in V \} $ has upper density equal to one.
Algebraic genericity for the set $U_{FM}(T)$ and the fact that $U_{FM}(T)$ and $X(T)$ are disjoint have been proven in \cite{BIEHLER1}.
So, in order to prove \cref{thm:wanted}, we will prove algebraic genericity for the set $X(T)$.

\section{A more general setting}
We will provide a more general proof of our result.
In particular we will replace $\C$ by any normed space over $\C$ and we will consider a generalized definition of harmonic functions introduced in \cite{BIEHLER1}.
However, in this section we will not assume that $E$ is a normed space as we will use some of the following results when considering a space in which the metric is not induced by a norm.

Let $E$ be a separable topological vector space over $\C$, that is metrizable with a metric $d$.
We consider the space $E^T$ with the metric
$$\rho(f,g) = \sum_{n=1}^{\infty}{\frac{1}{2^n}\frac{d(f(z_n), g(z_n))}{1+d(f(z_n), g(z_n))}},$$
where $f,g \in E^T$ and $\{z_n : n\in \N\}$ is a enumeration of $T$.
Since $E$ is separable and $T$ denumerable, the metric space $(E^T, \rho)$ is separable.
\begin{definition}
A function $f:T \to E$ is called generalized harmonic if and only if for all $x \in T$ it holds that
$$f(x) = \sum_{y \in S(x)}{w(x,y)f(y)},$$
where $w(x,y) \in \C\setminus \{0\}$ such that $\sum_{y \in S(x)}{w(x,y)} = 1$. The set of generalized harmonic functions shall be denoted by $H(T,E)$.
\end{definition}

From now we will drop the word generalized and refer to the elements of $H(T,E)$ simply as harmonic functions. Notice that the set $H(T,E)$ is a separable metric space as a metric subspace of the separable metric space $(E^T,\rho)$.

A function $h: \partial T \to E$ is $\mathcal{M}$-measurable if and only if for every open set $V \subseteq E$ we have that $h^{-1}(V) \in \mathcal{M}$.
By identifying almost everywhere equal $\mathcal{M}$-measurable functions we obtain the space $L^{0}(\partial T, E)$.
We endow $L^{0}(\partial T, E)$ with the metric
$$P(\psi,\phi) = \int_{\partial T}{\frac{d(\psi(e), \phi(e))}{1+d(\psi(e), \phi(e))}} \, d\P(e),$$
where $\psi,\phi \in L^{0}(\partial T, E)$.
The topology induced by $P$ is that of convergence in probability.
Because $E$ is separable and by construction of the measure space $(\partial T, \mathcal{M}, \P)$ there exists a dense sequence $\{h_n\}_{n=1}^{\infty}$ in $L^0(\partial T, E)$, such that each $h_n$ is $\mathcal{M}_{k(n)}$-measurable for some $k(n) \in \N$.

In general, the set of all measurable functions from a measurable space to some metrizable topological vector space needs not be a vector space.
However, in this setting we have the following proposition concerning the structure of $L^0(\partial T, E)$.

\begin{proposition}
\label{prop:struc}
Let $E$ be a separable topological vector space over $\C$, metrizable with a metric $d$.
The following are true.
\begin{enumerate}[(i)]
\item The space $L^0(\partial T, E)$ contains no isolated points.
\item The set $L^0(\partial T, E)$ when considered with pointwise addition and $\C$-scalar multiplication is a vector space.
\item If the metric $d$ of $E$ is translation invariant, then addition on $L^0(\partial T, E)$ is continuous.
\item If the metric $d$ of $E$ is induced by a norm, then scalar multiplication on $L^0(\partial T, E)$ is continuous.
\end{enumerate}
\end{proposition}

\begin{proof}
\begin{enumerate}[(i)]
\item This was proved as a Lemma in \cite{BIEHLER2}, without using any of the other claims stated in this proposition.

\item It suffices to prove that if $f:\partial T\to E$ and $g: \partial T \to E$ are two $\mathcal{M}$-measurable functions and $a \in \C$ is a complex number, then the functions $f + g$ and $af$ are $\mathcal{M}$-measurable.
Since the set $\{h_n : n \in \N\}$ is dense in $L^0(\partial T, E)$ we can find a sequence in that set that convergences to $f$ in probability.
We can then find a subsequence of that sequence that convergences to $f$ pointwise almost everywhere.
By reindexing we can find a sequence $\{f_n\}_{n=1}^{\infty}$ in the set $\{h_n : n \in \N\}$ such that $f_n \to f$ pointwise almost everywhere and each $f_n$ is $\mathcal{M}_{\tau(n)}$-measurable for some $\tau(n) \in \N$.
By the same argument we can find a sequence $\{g_n\}_{n=1}^{\infty}$ in the set $\{h_n : n \in \N\}$ such that $g_n \to g$ pointwise almost everywhere and each $g_n$ is $\mathcal{M}_{\ell(n)}$-measurable for some $\ell(n) \in \N$.
Thus we can find a set $A \in \mathcal{M}$ with $\P(A) = 1$ such that $f_n(e) \to f(e)$ and $g_n(e) \to g(e)$ for all $e \in A$.
Since addition on $E$ is continuous we have that $f_n(e) + g_n(e) \to f(e) + g(e)$ for all $e \in A$.
By setting $r(n) = \max\{\tau(n), \ell(n)\}$ for all $n \in \N$ we have that $\mathcal{M}_{\tau(n)} \subseteq \mathcal{M}_{r(n)} $ and $\mathcal{M}_{\ell(n)} \subseteq \mathcal{M}_{r(n)}$ for all $n \in \N$.
Therefore, both $f_n$ and $g_n$ are $\mathcal{M}_{r(n)}$-measurable for all $n \in \N$ and so they are constant on all the sets $B_x$ for $x \in T_{r(n)}$.
This implies that the sum $f_n + g_n$ is constant on all the sets $B_x$ for $x \in T_{r(n)}$ and thus, the sum $f_n +g_n$ is $\mathcal{M}_{r(n)}$-measurable and therefore $\mathcal{M}$-measurable.
So, we have that $f + g$ is an almost everywhere pointwise limit of the sequence of $\mathcal{M}$-measurable functions $\{f_n + g_n\}_{n=1}^{\infty}$, the values of which are elements of the metric space $(E,d)$, and the measure space $(\partial T, \mathcal{M}, \P)$ is complete.
This implies that the function $f+g$ is $\mathcal{M}$-measurable.
The proof that $af$ is $\mathcal{M}$-measurable is similar, it requires the fact that multiplication on $E$ is continuous, and is omitted.

\item The proof of this claim is similar to the proof of the respective claim about real-valued functions, which is well known.
It suffices to replace the absolute value in $\R$ by a translation invariant metric $d$. Thus, the proof is omitted.

\item The proof of this claim is similar to the proof of the respective claim about real-valued functions, which is well known.
It suffices to replace the absolute value in $\R$ by a norm induced metric $d$.
Thus, the proof is omitted.
\end{enumerate}
\end{proof}
For every $f \in H(T,E)$ and $n \in \N$ we define a function $\omega_n(f): \partial T \to E$ by $\omega_n(f)(e) = f(z)$, where $z$ is the unique element of $e \cap T_n$ and $e \in \partial T$.
Notice that $\omega_n(f)$ is constant on each $B_x$ for $x \in T_n$ and thus it is $\mathcal{M}_n$-measurable.
\begin{definition}
A harmonic function $f \in H(T,E)$ is called universal if and only if the sequence $\{\omega_n(f)\}_{n=1}^{\infty}$ is dense in $L^0(\partial T, E)$.
The set of universal harmonic functions shall be denoted with $U(T,E)$.
\end{definition}
The following theorem was proved in \cite{BIEHLER2}.
\begin{theorem}
Let $E$ be a separable metrizable topological vector space over $\C$. The set $U(T,E) \cup \{0\}$ contains a vector space which is dense in $H(T,E)$, that is we have algebraic genericity for the set $U(T,E)$.
\end{theorem}
We will now consider some classes of universal functions.
\begin{definition}
A harmonic function $f \in H(T,E)$ is called frequently universal if and only if for every nonempty open set $V \subseteq L^0(\partial T, E)$ the set $\{ n\in \N: \omega_n(f) \in V\}$ has strictly positive lower density. The set of frequently universal functions shall be denoted by $U_{FM}(T,E)$.
\end{definition}

\begin{definition}
A harmonic function $f \in H(T,E)$ is said to belong to the class $X(T,E)$ if and only if for every nonempty open set $V \subseteq L^0(\partial T,E)$ the set $\{ n\in \N: \omega_n(f) \in V\}$ has upper density equal to one.
\end{definition}

The following theorem was proved in \cite{BIEHLER2}.

\begin{theorem}
\label{thm:Xdens}
Let $E$ be a separable metrizable topological vector space over $\C$. The class $X(T,E)$ is $G_{\delta}$ and dense in $H(T,E)$.
\end{theorem}

Notice that $f \in X(T,E)$ if and only if for every nonempty non-dense open set $V \subseteq L^0(\partial T, E)$ the set $\{ n\in \N: \omega_n(f) \in V\}$ has lower density equal to zero.
From this observation and the existence of a nonempty non-dense open set $V \subseteq L^0(\partial T, E)$ we can deduce the following proposition, which was proved in a different way in \cite{BIEHLER1}.

\begin{proposition}
\label{prop:disjoint}
Let $E$ be a separable metrizable topological vector space over $\C$.
The sets $U_{FM}(T,E)$ and $X(T,E)$ are disjoint, that is $$U_{FM}(T,E) \cap X(T,E) = \varnothing.$$
\end{proposition}

\section{Algebraic genericity}

In this section we will assume that the metric $d$ of $E$ is induced by a norm as we want $L^0(\partial T, E)$ to be a topological vector space when considered with the topology of convergence in probability.
Under this assumption we can prove the following theorem, the proof of which can be found in \cite{BIEHLER1, BIEHLER2}.

\begin{theorem}
\label{thm:UFMalg}
Let $E$ be normed space over $\C$.
The set $U_{FM}(T,E)\cup \{0\}$ contains a vector space which is dense in $H(T,E)$.
\end{theorem}

We will now proceed to proving the same result for the class $X(T,E)$.
For this we will consider generalized harmonic functions $f:T \to E^{\N}$, where $E^{\N}$ is a separable metrizable topological vector space when considered with its product metric $\hat{d}$ which is translation invariant.
Notice that a function $f = (f_1,f_2,\dots)$ belongs to set $H(T,E^\N)$ if and only if the function $f_n$ belongs to the set $H(T,E)$ for all $n \in \N$.
By \cref{prop:struc} we know that $L^0(\partial T, E)$ is a topological vector space and $L^0(\partial T, E^{\N})$ is a vector space in which addition is continuous, each considered with the topology of convergence in probability.
We will now prove two lemmas that will lead to the result.

\begin{lemma}
\label{lem:linearity}
Let $E$ be normed space over $\C$.
If $f = (f_1,f_2,\dots) \in X(T,E^{\N})$ then the vector space $\Span\{f_n : n\in \N\}$ is contained in the set $X(T,E) \cup \{0\}$.
\end{lemma}
\begin{proof}
Let $a_1f_1 + \dots + a_sf_s$ be an arbitrary element of the set $\Span\{f_n : n\in \N\}$ with $a_s \neq 0$.
Let $V$ be a nonempty open subset of $L^0(\partial T, E)$.
There exist some $\e > 0$ and $h \in V$ such that $B(h,\e) \subseteq V$.
For $i \in \{1,\dots, s\}$ we define
$$
b_i =
\begin{cases}
a_i, &\text{if } a_i \neq 0 \\
1, &\text{if } a_i = 0
\end{cases}.
$$
We set $\delta = \dis{\frac{1}{s\cdot \e}}$ and consider the set
$$
\widehat{V} = \frac{1}{b_1}B(0,\delta) \times \frac{1}{b_2}B(0,\delta) \times \cdots \times \frac{1}{b_s}B(h,\delta) \times L^0(\partial T, E) \times \cdots.
$$
We now fix some $n \in \N$ and suppose that $\omega_n(f) \in \widehat{V}$.
Then for every $i \in \{1,\dots,s-1\}$ we have
$$\omega_n(f_i) \in \frac{1}{b_i}B(0,\delta) \text{~~~and~~~} \omega_n(f_s) \in \frac{1}{b_s}B(h,\delta).$$
By the definition of $b_1,\dots,b_n$ and the linearity of $\omega_n(\cdot)$ we conclude that for all $i \in \{1,\dots,s-1\}$ we have
$$P(\omega_n(a_if_i), 0) < \delta \text{~~~and~~~} P(\omega_n(a_sf_s), h) < \delta.$$
The metric $P$ is translation invariant, since the metric $d$ is translation invariant, thus
\begin{align*}
P(\omega_n(a_1f_1 + \dots + a_sf_s), h) &\leq P(\omega_n(a_1f_1), 0) + \dots + P(\omega_n(a_{s-1}f_{s-1}),0) + P(\omega_n(a_sf_s),h) \\
&< s\cdot \delta.
\end{align*}
Since $s\cdot \delta = \e$ we conclude that $\omega_n(a_1f_1+\dots+a_sf_s) \in B(h,\e)$, where $B(h, \e) \subseteq V$.
Therefore, we have proved that
$$\{n \in \N: \omega_n(f) \in \widehat{V} \} \subseteq \{n \in \N: \omega_n(a_1f_1 + \dots + a_sf_s) \in V\}.$$
Notice that convergence in probability in $L^0(\partial T,E^\N)$, which is the same set as $L^0(\partial T,E)^\N$, is equivalent to convergence in the product topology when considering $L^0(\partial T,E)$ with the topology of convergence in probability.
Therefore $\widehat{V}$ is an open subset of $L^0(\partial T,E^\N)$ and it is also nonempty.
Since $f \in X(T,E)$, the set $\{n \in \N: \omega_n(f) \in \widehat{V}\}$ has upper density equal to one, which implies that the set $\{n \in \N: \omega_n(a_1f_1 + \dots + a_sf_s) \in V\}$ has upper density equal to one.
Thus $a_1f_1 + \dots a_sf_s \in X(T,E)$.
\end{proof}

\begin{lemma}
\label{lem:density}
Let $E$ be normed space over $\C$.
There exists some sequence $\{f_n\}_{n=1}^{\infty}$ dense in $H(T,E)$ such that the function $f = (f_1,f_2,\dots) $ belongs to the class $X(T,E^\N)$.
\end{lemma}
\begin{proof}
The metric space $H(T,E)$ is separable, so let $\{\varphi_n: n \in \N\}$ be a dense sequence in $H(T,E)$.
Since $E^\N$ when considered with the product metric $\hat{d}$ is a separable topological vector space, \cref{thm:Xdens} applies.
Thus, the class $X(T,E^\N)$ is dense in $H(T,E^\N)$ and in particular nonempty, so let $f = (f_1,f_2,\dots)$ be an element of $X(T,E^\N)$.
We now fix some $n \in \N$.
There exists some $j_0(n)\in \N$ such that
$$\sum_{j= j_0(n)}^{\infty}{\frac{1}{2^j}} < \frac{1}{n}.$$
There also exists some $N(n) \in \N$ such that
$$\{z_0,\dots,z_{j_0(n)} \} \subseteq \bigcup_{k=0}^{N(n)}{T_k},$$
where $\{z_n:n\in \N\}$ is the denumeration of $T$ through which the metric $\rho$ was defined.
We set $g_n(x) = \varphi_n(x)-f_n(x)$ for all $x \in \bigcup_{k=0}^{N(n)}{T_k}$ and we set $g_n(x) = g_n(x^-)$ for all $x \in \bigcup_{k=N(n)+1}^{\infty}{T_k}$, where $x^-$ is the father of $x$.
Notice that the function $g_n: T \to E$ defined this way is a generalized harmonic function such that $\omega_k(g_n) = \omega_{N(n)}(g_n)$ for all $k > N(n)$ and $\rho(\varphi_n - f_n,g_n) < 1/n$.
Since the metric $d$ is translation invariant it follows that the metric $\rho$ is translation invariant, thus $\rho(f_n + g_n, g_n) < 1/n$ for all $n \in \N$.
Since $H(T,E)$ has no isolated points, the sequence $\{f_n + g_n\}_{n=1}^{\infty}$ is dense in $H(T,E)$.
We will now show that the function $f+g = (f_1+g_1,f_2+g_2,\dots)\in H(T,E^\N)$ belongs to the class $X(T,E^\N)$ which completes the proof.
It suffices to show that for every nonempty open basic set $V \subseteq L^0(\partial T,E)$ the set $\{n \in \N: \omega_n(f+g) \in V\}$ has upper density equal to one.
Since convergence in probability in $L^0(\partial T,E^\N)$, which is the same set as $L^0(\partial T,E)^\N$, is equivalent to convergence in the product topology when considering $L^0(\partial T,E)$ with the topology of convergence in probability, a nonempty basic open set of $L^0(\partial T,E^\N)$ can be written as
$$V = V_1 \times \cdots \times V_s \times L^0(\partial T, E) \times \cdots,$$
where $V_1,\dots,V_s$ are nonempty open subsets of $L^0(\partial T, E)$.
We set $L = \max\{N(1),\dots,N(s)\}$.
Then, for all $n \geq L$ and all $j \in \{1,\dots,s\}$ we have $\omega_n(g_j) = \omega_L(g_j)$.
We fix some $n \in\N$ with $n \geq L$ and suppose that $\omega_n(f) \in V - \omega_L(g)$.
Then $\omega_n(f_j) \in V_j - \omega_L(f_j)$ for all $j \in \{1,\dots,s\}$, which implies that $\omega_n(f_j)+ \omega_n(g_j) \in V_j$.
But $\omega_n(f_j + g_j) = \omega_n(f_j) + \omega_n(g_j)$ for all $j \in \{1,\dots,n\}$, so we have that $\omega_n(f_j+g_j) \in V_j$ for all $j \in \{1,\dots,n\}$, which in turn implies that $\omega_n(f+g) \in V$.
Thus, we have proved that
$$\{n \geq L: \omega_n(f) \in V - \omega_L(g)\} \subseteq \{n\in\N: \omega_n(f+g) \in V\}.$$
But the set $V - \omega_L(g)$ is a nonempty open subset of $L^0(\partial T, E^\N)$ and $f \in X(T,E^\N)$, therefore the set $\{n \in \N: \omega_n(f) \in V - \omega_L(g)\}$ has upper density equal to one.
This implies that the set $\{n \geq L: \omega_n(f) \in V - \omega_L(g)\}$ has upper density equal to one and therefore the set $\{n\in\N: \omega_n(f+g) \in V\}$ has upper density equal to one.
\end{proof}

We can now prove the following theorem.

\begin{theorem}
\label{thm:Xalg}
Let $E$ be normed space over $\C$.
The set $X(T,E)\cup\{0\}$ contains a vector space which is dense in $H(T,E)$.
\end{theorem}
\begin{proof}
By \cref{lem:density} there exists some $f = (f_1,f_2, \dots ) \in X(T,E^\N)$ such that the sequence $\{f_n\}_{n=1}^{\infty}$ is dense in $H(T,E)$.
By \cref{lem:linearity} the vector space $\Span\{f_n: n\in \N\}$ is contained in the set $X(T,E) \cup \{0\}$ and the sequence $\{f_n\}_{n=1}^{\infty}$ is contained in the vector space $\Span\{f_n: n\in \N\}$, which implies that vector space $\Span\{f_n: n\in \N\}$ is dense in $H(T,E)$.
\end{proof}

We can now prove the main result of this paper.

\begin{theorem}
Let $E$ be normed space over $\C$.
There exists two vector spaces $F_1,F_2$ contained in $U(T,E)\cup\{0\}$ which are dense in $H(T,E)$ such that $F_1 \cap F_2 = \{0\}$, that is we have double algebraic genericity for the set $U(T,E)$.
\end{theorem}
\begin{proof}
By \cref{thm:UFMalg} there exists a vector space $F_1$ contained in $U_{FM}(T,E)\cup\{0\}$, which is a subset of $U(T,E)\cup\{0\}$, that is dense in $H(T,E)$.
By \cref{thm:Xalg} there exists a vector space $F_2$ contained in $X(T,E)\cup\{0\}$, which is a subset of $U(T,E)\cup\{0\}$, that is dense in $H(T,E)$.
By \cref{prop:disjoint} we have that $F_1 \cap F_2 = \{0\}$.
\end{proof}

\section*{\fontsize{11}{15}\selectfont Acknowledgements}
I would like to thank V. Nestoridis for his assistance in my effort of understanding \cite{BIEHLER2} and \cite{BIEHLER1} and also for his remarks and advice concerning this paper.

\printbibliography
\bigskip
\noindent
C.A. Konidas\\
National and Kapodistrian University of Athens\\
Department of Mathematics,\\
e-mail:
\href{mailto:xkonidas@gmail.com}{\tt xkonidas@gmail.com}
\end{document}